\documentclass[reqno, 11pt]{amsart}
\usepackage{amsmath}
 \usepackage{amssymb}
\usepackage{amsthm}
\usepackage{times}
\usepackage{latexsym}
\usepackage[mathscr]{eucal}

\numberwithin{equation}{section}
 
  \newtheorem{theorem}{Theorem}[section]

  \newtheorem{remark}[theorem]{Remark}

  \newtheorem{example}[theorem]{Example}

\title[Some symmetry properties of four-dimensional Walker manifolds]{Some symmetry properties of four-dimensional Walker manifolds}

\author[Abdoul Salam Diallo, Fortun\'{e} Massamba ]{Abdoul Salam Diallo*, Fortun\'{e} Massamba**}

\newcommand{\acr}{\newline\indent}

\address{\llap{*\,} Universit\'e Alioune Diop de Bambey\acr
UFR SATIC, D\'epartement de Math\'ematiques\acr
B. P. 30, Bambey, S\'en\'egal}
\email{abdoulsalam.diallo@uadb.edu.sn}
\address{\llap{**\,} School of Mathematics, Statistics and Computer Science\acr
 University of KwaZulu-Natal\acr
 Private Bag X01, Scottsville 3209\acr
South Africa}
\email{massfort@yahoo.fr, Massamba@ukzn.ac.za}

\thanks{} 



\begin{document}
 
\begin{abstract}  
In this paper, we investigate geometric properties of some curvature tensors of a four-dimensional Walker manifold. Some characterization theorems are also obtained.
\end{abstract}

\maketitle

\section{Introduction} 

A Walker $n$-manifold is a pseudo-Riemannian manifold, which admits a field of parallel null
$r$-planes, with $r\leq \frac{n}{2}$. The canonical forms of the metrics were investigated
by A. G. Walker \cite{wa}. Of special interest are the even-dimensional Walker manifolds $(n=2m)$
with fields of parallel null planes of half dimension $(r=m)$.

It is known that Walker metrics have served as a powerful tool of constructing interesting
indefinite metrics which exhibit various aspects of geometric properties not given by any
positive definite metrics. Among these, the significant Walker manifolds are the examples
of the non-symmetric and non-homogeneous Osserman manifolds \cite{bro,cgm}. Recently, Banyaga and Massamba derived in \cite{bm} a Walker metric when studying the non-existence of certain Einstein metric on some 
symplectic manifolds.

Our purpose is to study a restricted $4$-Walker metrics by focusing on their curvature properties.
The main results of this paper are the characterization of Walker metrics which are Einstein, locally symmetric 
Einstein and locally conformally flat. The paper is organized as follows. In section \ref{canonical}, we recall some basic facts about Walker metrics by explicitly writing its Levi-Civita connection and the curvature tensor. Walker metrics which are Einstein are investigated in section \ref{Einstein} (Theorem \ref{thmEinstein}). In section \ref{einsteinwalker}, we study the Walker metrics which are locally symmetric Einstein (Theorem \ref{thmLocallySymmetric}). Finally, we discuss in section \ref{conformal}, the conformally locally flat property of Walker metric (Theorem \ref{thmLocallyConformallyFlat}).

\section{The canonical form of a Walker metric}\label{canonical}

Let $M$ be a pseudo-Riemannian manifold of signature $(n,n)$. We suppose given a splitting of the
tangent bundle in the form $TM=\mathcal{D}_1\oplus \mathcal{D}_2$ where $\mathcal{D}_1$ and 
$\mathcal{D}_2$ are smooth subbundles which are called distribution. This define two complementary 
projection $\pi_1$ and $\pi_2$ of $TM$ onto $\mathcal{D}_1$ and $\mathcal{D}_2$. We say that
$\mathcal{D}_1$ is parallel distribution if $\nabla \pi_1=0$. Equivalently this means that if 
$X_1$ is any smooth vector field taking values in $\mathcal{D}_1$, then $\nabla X_1$ again takes
values in $\mathcal{D}_1$. If $M$ is Riemannian, we can take $\mathcal{D}_2=\mathcal{D}^{\perp}_{1}$ 
to be the orthogonal complement of $\mathcal{D}_1$ and it that case $\mathcal{D}_2$ is again parallel. 
In the pseudo-Riemannian setting, $\mathcal{D}_1\cap \mathcal{D}_2$ need not be trivial. We say 
that $\mathcal{D}_1$ is a null parallel distribution if $\mathcal{D}_1$ is parallel and the
metric restricted to $\mathcal{D}_1$ vanish identically. Manifolds which admit null parallel 
distribution are called Walker manifolds.\\ 
A neutral $g$ on an $4$-manifold $M$ is said to be a Walker metric if there exists a $2$-dimensional
null distribution $\mathcal{D}$ on $M$ which is parallel with respect to $g$. From Walker
theorem \cite{wa}, there is a system of coordinates $(u_1,u_2,u_3,u_4)$ with respect to which 
$g$ takes the local canonical form
\begin{eqnarray*}
(g_{ij}) =
\left(
      \begin{array}{cc}
       0&I_2\\
       I_2&B
       \end{array}
\right),
\end{eqnarray*}
where $I_2$ is the $2\times 2$ identity matrix and $B$ is a symmetric $2\times 2$ matrix whose the 
coefficients are the functions of the $(u_1,\cdots,u_4)$. Note that $g$ is of neutral signature
$(++--)$ and that the parallel null $2$-plane $\mathcal{D}$ is spanned locally by 
$\{\partial_1,\partial_2\}$, where $\partial_i=\frac{\partial}{\partial_i}, i=1,2,3,4$.\\
Let $M_{a,b,c}:=(\mathcal{O},g_{a,b,c})$, where $\mathcal{O}$ be an open subset of $\mathbb{R}^4$ 
and $a,b,c\in \mathcal{C}^{\infty}(\mathcal{O})$ be smooth functions on $\mathcal{O}$, then
\begin{eqnarray*}
(g_{a,b,c})_{ij} =
\left(
      \begin{array}{cccccc}
       0&0&1&0\\
       0&0&0&1\\
       1&0&a&c\\
       0&1&c&b
      \end{array}
\right),
\end{eqnarray*}
where $a,b$ and $c$ are functions of the $(u_1,\cdots,u_4)$. We denote, $h_i=\frac{\partial h(u_1,\cdots,u_4)}{\partial u_i}$,
for any function $h(u_1,\cdots,u_4)$. In \cite{cgm}, Einsteinian, Osserman or locally conformally flat Walker manifolds 
were investigated in the restricted form of metric when $c(u_1,u_2,u_3,u_4)=0$. In this paper, following \cite{cgm}, we 
consider the specific Walker metrics on a $4$-dimensional manifold with
\begin{eqnarray}\label{e1}
 a=a(u_1,u_2),\quad b=b(u_1,u_2)\quad \mbox{and}\quad  c=c(u_1,u_2),
\end{eqnarray}
and investigate conditions for a Walker metric (\ref{e1}) to be Einsteinian, locally symmetric Einstein and 
locally conformally flat.\\
A straightforward calculation show that the Levi-Civita connection of a Walker metric (\ref{e1}) is given by
\begin{align*} \label{e2}
\nabla_{\partial_1} \partial_3 &=  \frac{1}{2}a_1 \partial_1 + \frac{1}{2}c_1 \partial_2,\quad \quad
\nabla_{\partial_1} \partial_4 = \frac{1}{2}c_1 \partial_1 + \frac{1}{2}b_1 \partial_2,\\ 
\nabla_{\partial_2} \partial_3 &=  \frac{1}{2}a_2 \partial_1 + \frac{1}{2}c_2 \partial_2,\quad \quad
\nabla_{\partial_2} \partial_4 = \frac{1}{2}c_2 \partial_1 + \frac{1}{2}b_2 \partial_2,\\ 
\nabla_{\partial_3} \partial_3 &=  \frac{1}{2}(aa_1 + ca_2)\partial_1
+ \frac{1}{2}(ba_2 + ca_1)\partial_2 -\frac{1}{2}a_1\partial_3 - \frac{1}{2}a_2\partial_4, 
\end{align*} 
\begin{align*} 
\nabla_{\partial_3} \partial_4 &=  \frac{1}{2}(ac_1+ cc_2)\partial_1 + \frac{1}{2}(bc_2 + cc_1)\partial_2
-\frac{1}{2}c_1\partial_3 - \frac{1}{2}c_2\partial_4,\nonumber\\
\nabla_{\partial_4} \partial_4 &=  \frac{1}{2}(ab_1 + cb_2)\partial_1
+ \frac{1}{2}(bb_2 + cb_1)\partial_2 -\frac{1}{2}b_1\partial_3 - \frac{1}{2}b_2\partial_4.
\end{align*}
From relations above, after a long but straightforward calculation we get that the nonzero components of the
$(0,4)$-curvature tensor of any Walker metric (\ref{e1}) are determined by
\begin{eqnarray}\label{e3}
R_{1313}&=& \frac{1}{2}a_{11}, \quad R_{1314}=\frac{1}{2}c_{11}, \quad R_{1323}=\frac{1}{2}a_{12},
\quad R_{1324}=\frac{1}{2}c_{12},\nonumber\\
R_{1334}&=& \frac{1}{4}(a_2b_1 - c_1c_2),\nonumber\\
R_{1414}&=& \frac{1}{2}b_{11},\quad R_{1423}=\frac{1}{2}c_{12},\quad R_{1424}=\frac{1}{2}b_{12},\nonumber\\
R_{1434}&=& \frac{1}{4}(c^{2}_{1} -a_1b_1 +b_1c_2 -b_2c_1)\nonumber\\
R_{2323}&=& \frac{1}{2}a_{22},\quad R_{2324}=\frac{1}{2}c_{22},\nonumber\\
R_{2334}&=& \frac{1}{4}(-c^{2}_{2} +a_2b_2 +a_1c_2-a_2c_1),\nonumber\\
R_{2424}&=& \frac{1}{2}b_{22},\quad R_{2434}=\frac{1}{4}(-a_2b_1 +c_1c_2),\nonumber\\
R_{3434}&=& \frac{1}{4}(ac^{2}_{1} +bc^{2}_{2} -aa_1b_1 -ca_1b_2 
- ca_2b_1 -ba_2b_2 +2cc_1c_2).
\end{eqnarray}
Next, let $\rho(X,Y)=\mathrm{trace}\{Z\longrightarrow R(X,Z)Y\}$ and $Sc=\mathrm{tr}(\rho)$, be the Ricci tensor and the scalar
curvature respectively. Then from (\ref{e3}) we have
\begin{eqnarray}\label{e4}
  \rho_{13} &=& \frac{1}{2}(a_{11}+c_{12}), \quad \rho_{14} = \frac{1}{2}(b_{12} +c_{11}),\nonumber\\
  \rho_{23} &=& \frac{1}{2}(a_{12}+c_{22}), \quad \rho_{24} = \frac{1}{2}(b_{22} +c_{12}),\nonumber\\
  \rho_{33} &=& \frac{1}{2}(-c^{2}_{2} +a_1c_2 +a_2b_2 -a_2c_1 +aa_{11} +2ca_{12}+ba_{22}),\nonumber\\
  \rho_{34} &=& \frac{1}{2}(-a_2b_1 +c_1c_2 +ac_{11} +2cc_{12} +bc_{22}),\nonumber\\
  \rho_{44} &=& \frac{1}{2}(-c^{2}_{1} +a_1b_1 -b_1c_2 +b_2c_1 +ab_{11} +2cb_{12} +bb_{22})
 \end{eqnarray}
 and
\begin{equation}\label{e5}
  Sc =\sum_{i,j=1}^{4} g^{ij} \rho_{ij} = a_{11} + b_{22} + 2c_{12}.
 \end{equation} 
 The nonzero components of the Einstein tensor $G_{ij} = \rho_{ij} -\frac{Sc}{4}g_{ij}$ are given by
 \begin{eqnarray}\label{e6}
  G_{13} &=& \frac{1}{4}a_{11} -\frac{1}{4}b_{22}, \quad G_{14} =\frac{1}{2}c_{11} + \frac{1}{2}b_{12},\nonumber\\
  G_{23} &=& \frac{1}{2}a_{12} +\frac{1}{2}c_{22}, \quad G_{24} = \frac{1}{4}b_{22} - \frac{1}{4}a_{11},\nonumber\\
  G_{33} &=& \frac{1}{4}aa_{11} +ca_{12} +\frac{1}{2}ba_{22} -\frac{1}{2}a_2c_1 +\frac{1}{2}a_1c_2\nonumber\\
  &+& \frac{1}{2}a_2b_2 -\frac{1}{2}c^{2}_{2} -\frac{1}{2}ac_{12} -\frac{1}{4}ab_{22},\nonumber\\
  G_{34} &=& \frac{1}{2}ac_{11} +\frac{1}{2}cc_{12} -\frac{1}{2}a_2b_1 +\frac{1}{2}c_1c_2 +\frac{1}{2}bc_{22}
  -\frac{1}{4}ca_{11} -\frac{1}{4}cb_{22},\nonumber\\
  G_{44} &=& \frac{1}{2}ab_{11} +cb_{12} -\frac{1}{2}c^{2}_{1} +\frac{1}{2}a_1b_1 -\frac{1}{2}b_1c_2 
  +\frac{1}{2}b_2c_1\nonumber\\
  &+& \frac{1}{4}bb_{22} -\frac{1}{4}ba_{11} -\frac{1}{2}bc_{12}.
 \end{eqnarray}

\section{Einstein Walker metrics}\label{Einstein}

We now turn our attention to be Einstein conditions for the Walker metric (\ref{e1}).
A Walker metric is said to be Einstein Walker metric if its Ricci tensor is a scalar multiple of 
the metric at each point i.e., there is a constant $\lambda$ so that
$$
\rho= \lambda g.
$$
The main result in this section is the following
\begin{theorem}\label{thmEinstein}
 A Walker metric (\ref{e1}) is Einstein if and only if the defining functions $a,b$ and $c$ are solution
 of the following PDES:
 \begin{eqnarray*}
 a_{11} &-& b_{22} =0, \quad b_{12} + c_{11} =0,\quad a_{12} + c_{22}=0, \nonumber\\
 a_1c_2 &+& a_2b_2 - a_2c_1 - c^{2}_{2} + 2ca_{12} + ba_{22} - ac_{12} =0,\nonumber\\
 a_2b_1 &-& c_1c_2 +  ca_{11} - ac_{11} - cc_{12} + bc_{22}  =0,\nonumber\\
 a_1b_1 &-& b_1c_2 + b_2c_1 - c^{2}_{1} + ab_{11} + 2cb_{12} - bc_{12} =0.
\end{eqnarray*}
\end{theorem}
\begin{proof}
From (\ref{e6}), the Einstein condition is equivalent to $G_{ij}=\rho_{ij}-\frac{Sc}{4}g_{ij}$. We get the following PDEs:
\begin{eqnarray}\label{e7}
 a_{11} &-& b_{22} =0, \quad b_{12} + c_{11} =0,\quad a_{12} + c_{22}=0, \nonumber\\
 a_1c_2 &+& a_2b_2 - a_2c_1 - c^{2}_{2} + 2ca_{12} + ba_{22} - ac_{12} =0,\nonumber\\
 a_2b_1 &-& c_1c_2 +  ca_{11} - ac_{11} - cc_{12} + bc_{22}  =0,\nonumber\\
 a_1b_1 &-& b_1c_2 + b_2c_1 - c^{2}_{1} + ab_{11} + 2cb_{12} - bc_{12} =0.
\end{eqnarray}
This system of partial differential equations (\ref{e7}) is hard to solve.
\end{proof}
In \cite{mm}, the authors apply the Lie symmetry group method to determine the Lie point symmetry group and provide example 
of solution of the system of partial differential equations (\ref{e7}).
\begin{example}\cite{mm} {\rm Let $(M,g_{a,b,c})$ be a Walker metric with 
\begin{eqnarray*}
a=-\frac{r_1}{r_2}e^{r_1u_1}e^{u_2}, b= -r_1r_2e^{r_1u_1}e^{u_2},\quad \mbox{and}\quad 
c=r_2e^{r_1u_1}e^{u_2}
\end{eqnarray*}
where $r_i$'s are arbitrary constants. Then the Walker metric $(M,g_{a,b,c})$ is Ricci flat and Einstein.
}
\end{example}

For the following restricted Walker metric 
\begin{eqnarray}\label{e8}
 a= a(u_1,u_2),\quad b= b(u_1,u_2)\quad \mbox{and}\quad  c=0
\end{eqnarray}
we have the following result.
\begin{theorem}\label{thmEinstein2}
A Walker metric (\ref{e8}) is Einstein if and only if the functions $a(u_1,u_2)$ and $b(u_1,u_2)$ are as follows:
\begin{eqnarray*}
a&=& a(u_1,u_2) = Ku^{2}_{1} + Au_1 + B(u_2),\\
b&=& b(u_1,u_2) = Ku^{2}_{2} + Cu_2 + D(u_1),
\end{eqnarray*}
where $K, A$ and $C$ are constants and $B,D$ are smooth functions satisfying the following PDE's:
\begin{eqnarray*}
B_2 D_1 &=&0,\\
(D_1 (u^{2}_{1}K + u_1A + B))_1 &=& 0,\\
(B_2 (u^{2}_{2}K + u_2C + D))_2 &=&0.
\end{eqnarray*}
\end{theorem}

\begin{proof}
The Einstein condition is equivalent to the following:
\begin{itemize}
 \item (i) $a_{12} =0$ and $b_{12} =0$;
 \item (ii) $a_{11}-b_{22} =0$;
 \item (iii) $a_2b_1  =0$;
 \item (iv) $a_1b_1 + ab_{11} =0$ and $a_2b_2 + ba_{22} =0$.
\end{itemize}
We divide the proof of the proposition into two steps.
\begin{enumerate}
 \item Step 1. The PDE system (i) imply that $a$ and $b$ take the following forms:
 \begin{eqnarray}\label{e9}
  a= \bar{a}(u_1) + \hat{a}(u_2)\quad \mbox{and}\quad
  b= \bar{b}(u_1) + \hat{b}(u_2).
 \end{eqnarray}
Substituting these functions $a$ and $b$ from (\ref{e9}) in the equation (ii), we get
\begin{eqnarray*}
 \bar{a}_{11}=\hat{b}_{22}.
\end{eqnarray*}
Therefore we have the following:
\begin{eqnarray*}
 \bar{a}_{11}(x) = K \quad \mbox{and}\quad \hat{b}_{22}(y)=K
\end{eqnarray*}
where $K$ is a constant. Then $\bar{a}$ (respectively $\hat{b}$) is a quadratic function of $u_1$ (respectively
$u_2$). Therefore we have the following
\begin{eqnarray}\label{e10}
 a= Ku^{2}_{1} + Au_1 + B(u_2) \quad b= Ku^{2}_{2} + Cu_2 + D(u_1)
\end{eqnarray}
where $A$ and $C$ are constants, $B=\hat{a}$ (respectively $D=\bar{b}$) are smooth functions of $u_2$ (respectively
$u_1$).
\item Step 2. The functions $a$ and $b$ in (\ref{e10}) satisfy the (i) and (ii) PDEs in the Einstein conditions. We 
must consider further conditions for $a$ and $b$ to satisfy the (iii) and (iv) PDE in the Einstein condition.
\begin{itemize}
 \item From the (iii) PDE's in the Einstein condition, we get the following condition:
\begin{eqnarray*}
 B_2 D_1 = 0.
\end{eqnarray*}
\item From (iv), the two equations $a_1b_1 + ab_{11} =0$ and $a_2b_2 + ba_{22}=0$ gives
\begin{eqnarray*}
\big( D_1 (u^{2}_{1}K + u_1A + B) \big)_1 &=& 0,\\
\Big( B_2 (u^{2}_{2}K + u_2C + D) \big)_2 &=&0,
\end{eqnarray*}
\end{itemize}
 \end{enumerate}
 which complete the proof.
\end{proof}
\begin{example}{\rm
Let us consider, a Walker metric $(M,g_{a,b,c})$ given by (\ref{e8}).
\begin{itemize}
\item If $b=0$ and $c=0$, then $(M,g_{a,b,c})$ is Einstein if and only if $a=Ku_1 +A$, where $K,A$ are
constant. Such a metric is Ricci flat.
\item If $a=0$ and $c=0$, then $(M,g_{a,b,c})$ is Einstein if and only if $b=Ku_2 +B$, where $K,B$ are
constant. Such a metric is Ricci flat.
\item The Walker metric with $a=Ku^{2}_{1}, b=Ku^{2}_{2}$ and $c=0$ is Einstein with non-zero scalar curvature 
equal to $4K$.
\end{itemize}
}
\end{example}
\begin{remark}{\rm
The description of the Einstein for the Walker metrics with $c=0$ can be also fund in \cite{cgm}, obtaining 
two different families for the scalar curvature being zero or not. However the problem in the general case 
remains open.}
\end{remark}
Four-dimensional Einstein Walker manifolds form the underling structure of many geometric and physical models
such as: \textit{hh}-space in general relativity, \textit{pp}-wave models and others areas.

\section{Locally symmetric Einstein-Walker metrics}\label{einsteinwalker}

This section deals with the locally symmetry of Einstein Walker manifolds. Let us consider the Einstein-Walker metric given in Theorem \ref{thmEinstein2},
\begin{eqnarray*}
a= Ku^{2}_{1} + Au_1 + B(u_2)\quad \mbox{and}\quad
b=  Ku^{2}_{2} + Cu_2 + D(u_1),
\end{eqnarray*}
where $K,A$ and $C$ are constants and $B,D$ are smooth functions satisfying the following PDE's:
\begin{eqnarray*}
B_2 D_1 &=&0,\\
(D_1 (u^{2}_{1}K + u_1A + B))_1 &=& 0,\\
(B_2 (u^{2}_{2}K + u_2C + D))_2 &=&0.
\end{eqnarray*}
\noindent
An Einstein-Walker manifold $(M,g_{a,b})$ is said to be locally symmetric if the curvature tensor $R$ of $(M,g_{a,b})$ satisfies
$
(\nabla_X R)(Y,Z,T,W)=0,
$
for any $X,Y,Z,T,W \in \Gamma(TM)$. By a straightforward calculation, we can see that the condition for the 
Einstein-Walker metric (Theorem \ref{thmEinstein2}) to be locally symmetric is equivalent to the following PDEs
\begin{eqnarray*}
a_1a_2b_2 &=&0, \quad a_1b_1b_2=0, \quad a_1a_{22}=0, \quad a_1b_{11}=0, \\
a_2b_{11} &=& 0, \quad b_1a_{22} =0, \quad b_2a_{22}=0, \quad b_2b_{11}=0, \\
aa_{11}b_1 &=& 0, \quad ba_2b_{22}=0.
\end{eqnarray*}
From the PDEs, We have the following results
\begin{theorem}\label{thmLocallySymmetric}
 A Walker metric given in Theorem \ref{thmEinstein2} is locally symmetric Einstein if and only if the functions $a(u_1,u_2)$ and
 $b(u_1,u_2)$ are constant.
\end{theorem}

\section{Locally conformally flat Walker metrics}\label{conformal}

Let $W$ denote the Weyl conformal curvature tensor given by
\begin{eqnarray*}
 W(X,Y,Z,T) :&=& R(X,Y,Z,T)\nonumber\\
 &+& \frac{Sc}{(n-1)(n-2)}\Big \{g(Y,Z)g(X,T)-g(X,Z)g(Y,T)\Big \}\nonumber\\
 &+& \frac{1}{n-2} \Big\{\rho(Y,Z)g(X,T) - \rho(X,Z)g(Y,T)\nonumber\\
 &-& \rho(Y,T)g(X,Z) +\rho(X,T)g(Y,Z) \Big\}.
\end{eqnarray*}

A pseudo-Riemannian manifold is locally conformally flat if and only if its Weyl tensor vanishes.
The nonzero components of Weyl tensor of the Walker metric defined by (\ref{e1}) are given by
\begin{eqnarray}\label{e11}
 W_{1313} &=& \frac{a_{11}}{6} + \frac{b_{22}}{6} - \frac{c_{12}}{6},\quad \quad
 W_{1314} = -\frac{b_{12}}{4} + \frac{c_{11}}{4},\nonumber\\
 W_{1323} &=& \frac{a_{12}}{4} - \frac{c_{22}}{4},\quad \quad
 W_{1324} = \frac{c_{12}}{2},\nonumber\\
 W_{1334} &=& \frac{ca_{11}}{12} -\frac{ab_{12}}{4} -\frac{cb_{22}}{6} +\frac{5cc_{12}}{12} +\frac{bc_{22}}{4},\nonumber\\
 W_{1414} &=& \frac{b_{11}}{2},\quad
 W_{1423} = -\frac{a_{11}}{12} -\frac{b_{22}}{12} +\frac{c_{12}}{3},\quad
 W_{1424} = \frac{b_{12}}{4} -\frac{c_{11}}{4};\nonumber\\
 W_{1434} &=& \frac{ba_{11}}{12} +\frac{ab_{11}}{4} +\frac{cb_{12}}{4} +\frac{bb_{22}}{12} -\frac{cc_{11}}{4}
 -\frac{bc_{12}}{12},\nonumber \\
 W_{2323} &=& \frac{a_{22}}{2},\quad
 W_{2324} = -\frac{a_{12}}{4} + \frac{c_{22}}{4},\nonumber\\
 W_{2334} &=& -\frac{aa_{11}}{12} -\frac{ca_{12}}{4} -\frac{ba_{22}}{4} -\frac{ab_{22}}{12} +\frac{ac_{12}}{12}
 -\frac{cc_{22}}{4},\nonumber\\
 W_{2424} &=& \frac{a_{11}}{6} +\frac{b_{22}}{6} -\frac{c_{12}}{6},\quad
 W_{2434} = \frac{ca_{11}}{6} + \frac{ba_{12}}{4} -\frac{cb_{22}}{12} -\frac{ac_{11}}{4} -\frac{5cc_{12}}{12},\nonumber
 \end{eqnarray}
 \begin{eqnarray}
 W_{3434} &=& \frac{c^2a_{11}}{6} +\frac{aba_{11}}{12} +\frac{bca_{12}}{2} +\frac{b^2a_{22}}{4}\nonumber\\
 &+& \frac{a^2b_{11}}{4} + \frac{acb_{12}}{2} +\frac{c^2b_{22}}{6} +\frac{abb_{22}}{12}\nonumber\\
 &-& \frac{acc_{11}}{2} -\frac{2c^2c_{12}}{3} -\frac{abc_{12}}{3} -\frac{bcc_{22}}{2}\nonumber\\
 &+& \frac{ba_1c_2}{4} -\frac{ca_1b_2}{4} +\frac{ca_2b_1}{4} -\frac{ba_2c_1}{4} -\frac{ab_1c_2}{4} +\frac{ab_2c_1}{4}.
\end{eqnarray}
Now it is possible to obtain the form of a locally conformall flat Walker metric as follows:

\begin{theorem}\label{thmLocallyConformallyFlat}
A Walker metric (\ref{e1}) is locally conformally flat if and only if the functions $a=a(u_1,u_2), b=b(u_1,u_2)$ and
 $c=c(u_1,u_2)$ are as follows
 \begin{eqnarray*}
  a&=& \frac{I}{2}u^{2}_{1} + Ju_1 + Eu_1u_2 + Fu_2 + K,\\
  b&=& -\frac{I}{2}u^{2}_{2} + Lu_2 + Mu_1u_2 + Nu_1 +R,\\
  c&=& \frac{M}{2}u^{2}_{1} + Pu_1 + \frac{E}{2}u^{2}_{2} + Gu_2 +(Q+H),
 \end{eqnarray*}
where the constants $E,F,G,H,I,J,K,L,M,N,P,Q$ and $R$ satisfy the following relations
\begin{eqnarray*}
0 &=& EN - JM +IP, \\
0 &=& EL - FM +IG, \\
0 &=& ER - KM +I(H+Q),\\
0 &=& K(LP-NG) + R(JG-FP) + (Q+H)(FN-JL).
\end{eqnarray*}
\end{theorem}
\begin{proof}
Sine the locally conformally flat is equivalently to the vanishing of the Weyl tensor, let consider (\ref{e11}) 
as a system of PDEs. We will prove the theorem in three steps.
\begin{enumerate}
\item Step 1. Considering the following components of the Weyl tensor of (\ref{e11}):
\begin{eqnarray}\label{e12}
W_{1324} = \frac{c_{12}}{2}=0, \quad
W_{1414} = \frac{b_{11}}{2}=0 \quad \mbox{and} \quad
W_{2323} = \frac{a_{22}}{2}=0.
\end{eqnarray}
The PDEs (\ref{e12}) imply that the functions $a,b$ and $c$ take the form
\begin{eqnarray*}
a(u_1,u_2)&=&u_2\bar{a}(u_1) + \hat{a}(u_1),\\
b(u_1,u_2)&=&u_1\bar{b}(u_2) + \hat{b}(u_2),\\
c(u_1,u_2)&=&\bar{c}(u_1) + \hat{c}(u_2).
\end{eqnarray*}
Considering the result of the step 1, the Weyl equations of (\ref{e11}) reduce to
\begin{eqnarray}\label{e13}
 W_{1313} &=& \frac{1}{6}(a_{11}+b_{22}),\quad
 W_{1314} = \frac{1}{4}(-b_{12}+c_{11}),\quad
 W_{1323} = \frac{1}{4}(a_{12}-c_{22}),\nonumber\\
 W_{1334} &=& \frac{ca_{11}}{12} -\frac{ab_{12}}{4} -\frac{cb_{22}}{6} +\frac{bc_{22}}{4},\quad
 W_{1423} = -\frac{1}{12}(a_{11}+b_{22}),\nonumber\\
 W_{1424} &=& \frac{1}{4}(b_{12}-c_{11}),\quad
 W_{1434} = \frac{b}{12}(a_{11}+b_{22}) +\frac{c}{4}(b_{12}-c_{11}),\nonumber\\
 W_{2324} &=& \frac{1}{4}(-a_{12}+c_{22}),\quad
 W_{2334} = -\frac{a}{12}(a_{11}+b_{22}) -\frac{c}{4}(a_{12}-c_{22}),\nonumber\\
 W_{2424} &=& \frac{1}{6}(a_{11}+b_{22}),\quad
 W_{2434} = \frac{ca_{11}}{6} + \frac{ba_{12}}{4} -\frac{cb_{22}}{12} -\frac{ac_{11}}{4},\nonumber\\
 W_{3434} &=& \frac{c^2}{6}(a_{11}+b_{22}) +\frac{ab}{12}(a_{11}+b_{22}) +\frac{bc}{2}(a_{12}-c_{22})
 + \frac{ac}{2}(b_{12}-c_{11}) \nonumber\\
 &+& \frac{ba_1c_2}{4} -\frac{ca_1b_2}{4} +\frac{ca_2b_1}{4} -\frac{ba_2c_1}{4} -\frac{ab_1c_2}{4} +\frac{ab_2c_1}{4}.
\end{eqnarray}
\item Step 2. Considering the following components of the PDEs (\ref{e13}):
\begin{eqnarray}\label{e14}
 W_{1313} = 0,\quad
 W_{1314} = 0,\quad \mbox{and}\quad
 W_{1323} = 0.
\end{eqnarray}
The PDEs (\ref{e14}) imply that the functions $a,b$ and $c$ take the form
\begin{eqnarray*}
 a&=& u_2(Eu_1 + F) + \frac{I}{2}u^{2}_{1} + Ju_1 + K,\\
 b&=& u_1(Mu_2 + N) - \frac{I}{2}u^{2}_{2} + Lu_2 + R, \\
 c&=& \frac{M}{2}u^{2}_{1} + Pu_1 + Q + \frac{E}{2}u^{2}_{2} + Gu_2 + H. 
\end{eqnarray*}
 Considering the result of the step 2, the Weyl equations (\ref{e13}) reduce to
\begin{eqnarray}\label{e15}
 W_{1334} &=& \frac{ca_{11}}{12} -\frac{ab_{12}}{4} -\frac{cb_{22}}{6} +\frac{bc_{22}}{4},\nonumber\\
 W_{2434} &=& \frac{ca_{11}}{6} + \frac{ba_{12}}{4} -\frac{cb_{22}}{12} -\frac{ac_{11}}{4},\nonumber\\
 W_{3434} &=& \frac{ba_1c_2}{4} -\frac{ca_1b_2}{4} +\frac{ca_2b_1}{4} -\frac{ba_2c_1}{4} -\frac{ab_1c_2}{4} +\frac{ab_2c_1}{4}.
\end{eqnarray}
\item Step 3. From (\ref{e15}), after some straightforward calculations, the following PDES
\begin{eqnarray*}
 W_{1334} =0,\quad W_{2434} = 0\quad \mbox{and}\quad W_{3434} = 0
\end{eqnarray*}
gives
\begin{eqnarray*}
0 &=& EN - JM +IP, \\
0 &=& EL - FM +IG, \\
0 &=& ER - KM +I(H+Q),\\
0 &=& K(LP-NG) + R(JG-FP) + (Q+H)(FN-JL).
\end{eqnarray*}
\end{enumerate}
This finish the proof.
\end{proof}
\begin{remark}{\rm
 From (\ref{e5}) and Theorem \ref{thmLocallyConformallyFlat}, we see that the locally conformally flat metric (\ref{e1}) 
 has vanishing scalar curvature.}
\end{remark}

%\section*{Acknowledgments}

%The authors also thank the referee for helping them to improve the presentation.

\end{document}